\documentclass[12pt]{amsart}

\usepackage{amsthm, amsmath, amsfonts, amssymb, graphicx, tikz-cd, enumerate}

\newtheorem{theorem}{Theorem}[section]
\newtheorem{lemma}[theorem]{Lemma}
\newtheorem{proposition}[theorem]{Proposition}

\theoremstyle{definition}
\newtheorem{definition}[theorem]{Definition}

\newtheorem{remark}[theorem]{Remark}

\newcommand{\func}[1]{\operatorname{#1}}

\def\lalg{\boldsymbol{\ell\hspace{-.2mm}\mathit{alg}}}
\def\bal{\boldsymbol{\mathit{ba}\ell}}
\def\ubal{\boldsymbol{\mathit{uba}\ell}}

\def\balg{\boldsymbol{\mathit{balg}}}
\def\KHaus{\sf{KHaus}}
\newcommand\caba{\sf{CABA}}
\def\Set{{\sf{Set}}}
\def\WSet{{\sf{WSet}}}
\def\Stone{{\sf{Stone}}}
\def\VL{{\sf{VL}}}
\def\pVL{{\sf{pVL}}}
\def\uVL{{\sf{uVL}}}

\def\Id{\func{Id}}

\title{Free bounded archimedean $\ell$-algebras}
\author{G.~Bezhanishvili, L.~Carai, P.~J.~Morandi}
\date{}

\address{Department of Mathematical Sciences, New Mexico State University, Las Cruces NM 88003 USA}
\email{guram@nmsu.edu, lcarai@nmsu.edu, pmorandi@nmsu.edu}

\subjclass[2010]{06F25; 13J25; 06F20; 46A40; 08B20; 54C30}
\keywords{Bounded archimedean $\ell$-algebra, Gelfand duality, freeness}

\begin{document}

\begin{abstract}
We show that free objects on sets do not exist in the category $\bal$ of bounded archimedean $\ell$-algebras. On the other hand, we introduce the category of weighted sets and prove that free objects on weighted sets do exist in $\bal$. We conclude by discussing several consequences of this result.
\end{abstract}

\maketitle

\section{Introduction}

The category $\bal$ of bounded archimedean $\ell$-algebras plays an important role in the study of Gelfand duality as algebraic counterparts of compact Hausdorff spaces live in $\bal$. Indeed, for each compact Hausdorff space $X$, the $\ell$-algebra $C(X)$ of continuous real-valued functions on $X$ is an object of $\bal$, and these algebras can be characterized as uniformly complete objects of $\bal$ (see Section~2 for details). This yields a contravariant functor $C$ from the category $\KHaus$ of compact Hausdorff spaces to $\bal$. The functor $C$ has a contravariant adjoint $Y:\bal\to\KHaus$ sending each $A\in\bal$ to the Yosida space $Y_A$ of maximal $\ell$-ideals of $A$ (more details are given in Section~2). This yields a contravariant adjunction between $\bal$ and $\KHaus$ that restricts to a dual equivalence between $\KHaus$ and the reflective subcategory $\ubal$ of $\bal$ consisting of uniformly complete objects of $\bal$. The reflector $\bal\to\ubal$ is the uniform completion functor. We thus arrive at the following commutative diagram.

\[
\begin{tikzcd}
\ubal \arrow[rr, hookrightarrow] && \bal \arrow[dl, "Y"]  \arrow[ll, bend right = 20] \\
&  \KHaus \arrow[ul,  "C"] &
\end{tikzcd}
\]

Gelfand duality can be thought of as a generalization to $\KHaus$ of Stone duality between the categories $\sf BA$ of boolean algebras and $\Stone$ of Stone spaces. By Tarski duality, the category $\caba$ of complete and atomic boolean algebras and complete boolean homomorphisms is dually equivalent to the category $\Set$ of sets and functions. A version of Tarski duality was established in \cite{BMO20b} between $\Set$ and a (non-full) subcategory $\balg$ of $\bal$ whose objects are Dedekind complete objects of $\bal$ whose boolean algebra of idempotents is atomic (see Section~4 for details). As we will see in Section~4, $\balg$ is a reflective subcategory of $\bal$, and the reflector is the canonical extension functor developed in \cite{BMO18c}.

In this article we study free objects in $\bal$ as well as in $\ubal$ and $\balg$. We first show that the forgetful functor $\bal\to\Set$ does not have a left adjoint, and hence free objects do not exist in $\bal$ in the usual sense. We next introduce the category $\WSet$ of weighted sets and prove that the forgetful functor $\bal\to\WSet$ does indeed have a left adjoint $F : \WSet \to \bal$, thus showing that free objects do exist in $\bal$ in this modified sense. As a consequence, we obtain that $F$ composed with the uniform completion functor is left adjoint to the forgetful functor $\ubal\to\WSet$, and that $F$ composed with the canonical extension functor is left adjoint to the forgetful functor $\balg\to\WSet$. Thus, free objects also exist in $\ubal$ and $\balg$ in this modified sense.

\section{Preliminaries}

We start by recalling some basic facts about lattice-ordered rings and algebras. We use Birkhoff's book \cite[Ch.~XIII and onwards]{Bir79} as our main reference. All rings we consider are assumed to be commutative and unital.

\begin{definition}
A ring $A$ with a partial order $\le$ is a \emph{lattice-ordered ring}, or an \emph{$\ell$-ring} for short, provided
\begin{itemize}
\item $(A,\le)$ is a lattice;
\item $a\le b$ implies $a+c \le b+c$ for each $c$;
\item $0 \leq a, b$ implies $0 \le ab$.
\end{itemize}
An $\ell$-ring $A$ is an \emph{$\ell$-algebra} if it is an $\mathbb R$-algebra and for each $0 \le a\in A$ and $0\le r\in\mathbb R$ we have $0 \le r\cdot a$.
\end{definition}

It is well known and easy to see that the conditions defining $\ell$-algebras are equational, and hence $\ell$-algebras form a variety. We denote this variety and the corresponding category of $\ell$-algebras and unital $\ell$-algebra homomorphisms by $\lalg$.

\begin{definition} \label{def: bal}
Let $A$ be an $\ell$-ring.
\begin{itemize}
\item $A$ is \emph{bounded} if for each $a \in A$ there is $n \in \mathbb{N}$ such that $a \le n\cdot 1$ (that is, $1$ is a \emph{strong order unit}).
\item $A$ is \emph{archimedean} if for each $a,b \in A$, whenever $n\cdot a \le b$ for each $n \in \mathbb{N}$, then $a \le 0$.
\end{itemize}
\end{definition}

Let $\bal$ be the full subcategory of $\lalg$ consisting of bounded archimedean $\ell$-algebras. It is easy to see that $\bal$ is not a variety (it is closed under neither products nor homomorphic images).

\begin{definition}
Let $A\in\lalg$. For $a\in A$, define the \emph{absolute value} of $a$ by
\[
|a|=a\vee(-a).
\]
If in addition $A \in \bal$, define the \emph{norm} of $a$ by
\[
||a||=\inf\{r\in\mathbb R \mid |a|\le r\cdot 1\}.
\]
Then $A$ is \emph{uniformly complete} if the norm is complete.
\end{definition}

\begin{remark}
Since $A \in \bal$ is bounded, $\|\!\cdot \!\|$ is well defined, and $\|\!\cdot\!\|$ is a norm since $A$ is archimedean.
\end{remark}

Let $\ubal$ be the full subcategory of $\bal$ consisting of uniformly complete $\ell$-algebras.

\begin{theorem} [{Gelfand duality}] \label{thm: Gelfand}
There is a dual adjunction between $\bal$ and $\KHaus$ which restricts to a dual equivalence between $\KHaus$ and $\ubal$.
\end{theorem}

\begin{remark}
Gelfand duality is also known as Gelfand-Naimark-Stone duality (see, e.g., \cite{BMO13a}). This duality was established by Gelfand and Naimark \cite{GN43} between $\KHaus$ and the category of commutative $C^*$-algebras. Gelfand and Naimark worked with complex-valued functions and associated with each $X\in\KHaus$ the $C^*$-algebra of all continuous complex-valued functions on $X$. On the other hand, Stone \cite{Sto40} worked with real-valued functions and associated with each $X\in\KHaus$ the $\ell$-algebra of all continuous real-valued functions on $X$. In this respect, Theorem~\ref{thm: Gelfand} is more closely related to Stone's work. Nevertheless, we follow Johnstone \cite[Sec.~IV.4]{Joh82} in calling this result Gelfand duality. The Gelfand-Naimark and Stone approaches are equivalent in that the complexification functor establishes an equivalence between $\ubal$ and the category of commutative $C^*$-algebras (see \cite[Sec.~7]{BMO13a} for details). \end{remark}

We briefly describe the functors $C:\KHaus\to\bal$ and $Y:\bal\to\KHaus$ establishing the dual adjunction of Theorem~\ref{thm: Gelfand}; for details see \cite[Sec.~3]{BMO13a} and the references therein. For a compact Hausdorff space $X$ let $C(X)$ be the ring of (necessarily bounded) continuous real-valued functions on $X$. For a continuous map $\varphi:X\to Y$ let $C(\varphi):C(Y)\to C(X)$ be defined by $C(\varphi)(f)=f\circ\varphi$ for each $f\in C(Y)$. Then $C:\KHaus\to\bal$ is a well-defined contravariant functor.

For $A \in \lalg$, we recall that an ideal $I$ of $A$ is an \emph{$\ell$-ideal} if $|a|\le|b|$ and $b\in I$ imply $a\in I$, and that $\ell$-ideals are exactly the kernels of $\ell$-algebra homomorphisms. If $A\in\bal$, then we can associate to $A$ a compact Hausdorff space as follows. Let $Y_A$ be the space of maximal $\ell$-ideals of $A$, whose closed sets are exactly sets of the form
\[
Z_\ell(I) = \{M\in Y_A\mid I\subseteq M\},
\]
where $I$ is an $\ell$-ideal of $A$. As follows from the work of Yosida \cite{Yos41}, $Y_A \in \KHaus$. The space $Y_A$ is often referred to as the \emph{Yosida space} of $A$. We set $Y(A)=Y_A$, and for a morphism $\alpha$ in $\bal$ we let $Y(\alpha)=\alpha^{-1}$. Then $Y:\bal\to\KHaus$ is a well-defined contravariant functor, and the functors $C$ and $Y$ yield a contravariant adjunction
between $\bal$ and $\KHaus$.

Moreover, for $X\in\KHaus$ we have that $\varepsilon_X:X\to Y_{C(X)}$ is a homeomorphism where
\[
\varepsilon_X(x)=\{f\in C(X) \mid f(x)=0\}.
\]
Furthermore, for $A\in\bal$ define $\zeta_A :A\to C(Y_A)$ by $\zeta_A(a)(M)=r$ where $r$ is the unique real number
satisfying $a+M=r+M$. Then $\zeta_A$ is a monomorphism in $\bal$ separating points of $Y_A$. Therefore, by the Stone-Weierstrass theorem, $\zeta_A : A \to C(Y_A)$ is the uniform completion of $A$. Thus, if $A$ is
uniformly complete, then $\zeta_A$ is an isomorphism. Consequently, the contravariant adjunction restricts to a dual equivalence between $\ubal$ and $\KHaus$, yielding Gelfand duality. Another consequence of these considerations is the following well-known result.

\begin{proposition}\label{prop: SW}
$\ubal$ is a full reflective subcategory of $\bal$, and the reflector assigns to each $A \in \bal$ its uniform completion $C(Y_A) \in \ubal$.
\end{proposition}

\section{Free Objects in $\bal$}

As we pointed out in Section~2, $\lalg$ is a variety, hence has free algebras by Birkhoff's theorem (see, e.g., \cite[Thm.~10.12]{BS81}). Since $\bal$ is not a subvariety of $\lalg$, it does not follow immediately that $\bal$ has free algebras. In fact, we show that free algebras on sets do not exist in $\bal$. In other words, we show that
the forgetful functor $U : \bal \to \Set$ does not have a left adjoint.

Let $A \in \lalg$. If $A \ne 0$, then sending $r \in \mathbb{R}$ to $r \cdot 1 \in A$ embeds $\mathbb{R}$ into $A$, and we identify $\mathbb{R}$ with a subalgebra of $A$. By this identification, if $A,B \ne 0$ and $\alpha : A \to B$ is a $\lalg$-morphism, then $\alpha(r) = r$ for each $r \in \mathbb{R}$.

\begin{lemma} \label{lem: norm preserving}
Let $A, B \in \bal$ and $\alpha : A \to B$ be a $\bal$-morphism. Then for each $a \in A$ we have $\alpha(|a|) = |\alpha(a)|$ and $\|\alpha(a)\| \le \|a\|$.
\end{lemma}

\begin{proof}
Let $a \in A$. Then $\alpha(|a|) = \alpha(a \vee -a) = \alpha(a) \vee -\alpha(a) = |\alpha(a)|$. For the second statement it is sufficient to assume $A, B \ne 0$. Since $|a| \le \|a\|$, we have $\alpha(|a|) \le \alpha(\|a\|) = \|a\|$. Therefore, $|\alpha(a)| = \alpha(|a|) \le \|a\|$ and hence $\|\alpha(a)\| \le \|a\|$.
\end{proof}

\begin{theorem} \label{lem: no frees}
The forgetful functor $U : \bal \to \Set$ does not have a left adjoint.
\end{theorem}

\begin{proof}
If $U$ has a left adjoint, then for each $X\in\Set$, there is $F(X)\in\bal$ and a function $f : X \to F(X)$ such that for each $A\in\bal$ and each function $g: X\to A$ there is a unique $\bal$-morphism $\alpha:F(X)\to A$ satisfying $\alpha\circ f=g$.
\[
\begin{tikzcd}[column sep = 5pc]
X \arrow[r, "f"] \arrow[rd, "g"'] & F(X) \arrow[d, "\alpha"] \\
& A
\end{tikzcd}
\]
Let $X$ be a nonempty set. Pick $x \in X$, choose $r \in \mathbb{R}$ with $r > \|f(x)\|$, and define $g : X \to \mathbb{R}$ by setting $g(y) = r$ for each $y \in X$. There is a (unique) $\bal$-morphism $\alpha : F(X) \to \mathbb{R}$ with $\alpha\circ f = g$, so
$\alpha(f(x)) = r$. But if $a \in F(X)$, then $\|\alpha(a)\| \le \|a\|$ by Lemma~\ref{lem: norm preserving}. Therefore,
\[
r = \|\alpha(f(x))\| \le \|f(x)\| < r.
\]
The obtained contradiction proves that $F(X)$ does not exist. Thus, $U$ does not have a left adjoint.
\end{proof}

The key reason for nonexistence of a left adjoint to the forgetful functor $U : \bal \to \Set$ can be explained as follows. The norm on $A$ provides a weight function on the set $A$, and each $\bal$-morphism $\alpha$ respects this weight function due to the inequality $\|\alpha(a)\| \le \|a\|$. The forgetful functor $U : \bal \to \Set$ forgets this, which is the obstruction to the existence of a left adjoint as seen in the proof of Theorem~\ref{lem: no frees}. We repair this by working with weighted sets.

\begin{definition}
\begin{itemize}
\item[]
\item A {\em weight function} on a set $X$ is a function $w$ from $X$ into the nonnegative real numbers.
\item A \emph{weighted set} is a pair $(X, w)$ where $X$ is a set and $w$ is a weight function on $X$.
\item Let $\WSet$ be the category whose objects are weighted sets and whose morphisms are functions $f : (X_1, w_1) \to (X_2, w_2)$ satisfying $w_2(f(x)) \le w_1(x)$ for each $x \in X$.
\end{itemize}
\end{definition}

\begin{lemma}
There is a forgetful functor $U:\bal\to\WSet$.
\end{lemma}

\begin{proof}
If $A \in \bal$, then
$(A, \|\! \cdot \!\|) \in \WSet$. Moreover, if $\alpha : A \to B$ is a $\bal$-morphism, then $\|\alpha(a)\| \le \|a\|$ by Lemma~\ref{lem: norm preserving}. Therefore, $\alpha$ is a $\WSet$-morphism. Thus, the assignment $A \mapsto (A, \|\!\cdot \!\|)$ defines a forgetful functor $U:\bal\to\WSet$.
\end{proof}

\begin{definition}
Let $A\in\lalg$. Call $a\in A$ {\em bounded} if there is $n \in \mathbb{N}$ with $-n\cdot 1 \le a \le n\cdot 1$. Let $A^*$ be the set of bounded elements of $A$.
\end{definition}

Let $A \in \lalg$. If $a,b \in A^*$, then there are $n,m \in \mathbb{N}$ with $-n\cdot 1 \le a \le n\cdot 1$ and $-m\cdot 1 \le b \le m\cdot 1$. Therefore, $-(n + m)\cdot 1 \le a \pm b \le (n + m)\cdot 1$. Similar facts hold for join, meet, and multiplication. Thus, we have the following:

\begin{lemma}\label{lem: bounded sub}
Let $A\in\lalg$. Then $A^*$ is a subalgebra of $A$, and hence $A^*$ is a bounded $\ell$-algebra. Therefore, if $A$ is archimedean, then $A^*\in\bal$.
\end{lemma}

Let $A\in\lalg$. As we pointed out in Section~2, $\ell$-ideals are kernels of $\ell$-algebra homomorphisms. However, if $I$ is an $\ell$-ideal of $A$, then the quotient $A/I$ may not be archimedean even if $A$ is archimedean.

\begin{definition}
We call an $\ell$-ideal $I$ of $A \in \lalg$ \emph{archimedean} if $A/I$ is archimedean.
\end{definition}

\begin{remark}
Archimedean $\ell$-ideals were studied by Banaschewski (see \cite[App.~2]{Ban97}, \cite{Ban05a}) in the category of archimedean $f$-rings.
\end{remark}

It is easy to see that the intersection of archimedean $\ell$-ideals is archimedean. Therefore, we may talk about the archimedean $\ell$-ideal of $A$ generated by $S \subseteq A$.

\begin{theorem} [Main result] \label{thm: UMP holds}
The forgetful functor $U : \bal \to \WSet$ has a left adjoint.
\end{theorem}

\begin{proof}
It is enough to show that there is a free object in $\bal$ on each $(X,w) \in\WSet$ (see, e.g., \cite[Ex.~18.2(2)]{AHS06}). Let $G(X)$ be the free object in $\lalg$ on $X$ and let $g:X\to G(X)$ be the corresponding map. We next quotient $G(X)$ by an archimedean $\ell$-ideal $I$ so that $-w(x) \le g(x) + I \le w(x)$ for each $x \in X$. Let $I$ be the archimedean $\ell$-ideal of $G(X)$ generated by
\[
\{ g(x) - ((g(x) \vee -w(x))\wedge w(x)) \mid x \in X\},
\]
and set $F(X,w) = G(X)/I$. Let $\pi : G(X) \to F(X,w)$ be the canonical projection. Clearly $F(X,w)$ is an archimedean $\ell$-algebra. We show that $F(X,w)$ is bounded, and hence that $F(X,w)\in\bal$.
Let $G(X)^*$ be the bounded subalgebra of $G(X)$ (see Lemma~\ref{lem: bounded sub}). Since $G(X)$ is generated by $\{g(x) \mid x \in X\}$, we have that $G(X)/I$ is generated by $\{\pi g(x) \mid x \in X\}$. Now,
\[
\pi g(x) = \pi((g(x) \vee -w(x))\wedge w(x))
\]
since $g(x) - ((g(x) \vee -w(x)) \wedge w(x)) \in I$. We have $-w(x) \le (g(x) \vee -w(x)) \wedge w(x) \le w(x)$, so $(g(x) \vee -w(x)) \wedge w(x) \in G(X)^*$. This shows that the generators of $F(X,w)$ lie in $\pi[G(X)^*]$, so $F(X,w) \cong G(X)^*/(I\cap G(X)^*)$ is a quotient of $G(X)^*$. Thus, $F(X,w)$ is bounded.

Let $f:X\to F(X,w)$ be given by $f(x)=\pi g(x)$. Since $f(x) = \pi((g(x) \vee -w(x))\wedge w(x))$, we have $-w(x) \le f(x) \le w(x)$, so $\|f(x)\| \le w(x)$. Therefore, $f$ is a $\WSet$-morphism.

 Let $A\in \bal$ and $h:X\to A$ be a $\WSet$-morphism, so
$\|h(x)\| \le w(x)$ for each $x \in X$.
There is an $\ell$-algebra homomorphism $\alpha : G(X) \to A$ with $\alpha\circ g = h$. Because $A$ is archimedean, $G(X)/\ker(\alpha)$ is archimedean, so $\ker(\alpha)$ is an archimedean $\ell$-ideal of $G(X)$. We show that $I \subseteq \ker(\alpha)$. It suffices to show that $g(x) - ((g(x) \vee -w(x)) \wedge w(x)) \in \ker(\alpha)$ for each $x \in X$ since $\ker(\alpha)$ is an archimedean $\ell$-ideal. Because $\|h(x) \| \le w(x)$, we have $-w(x) \le h(x) \le w(x)$. Therefore,
\begin{align*}
\alpha((g(x) \vee -w(x)) \wedge w(x)) &= (\alpha g(x) \vee -w(x)) \wedge w(x) \\
&= (h(x) \vee -w(x)) \wedge w(x) \\
&= h(x) \\
&= \alpha g(x),
\end{align*}
and hence $\alpha(g(x) - ((g(x) \vee -w(x)) \wedge w(x))) = 0$. Thus, $I \subseteq \ker(\alpha)$, so there is a well-defined $\ell$-algebra homomorphism $\overline{\alpha} : F(X,w) \to A$ satisfying $\overline{\alpha} \circ \pi = \alpha$. Consequently, $\overline{\alpha}\circ f = \overline{\alpha}\circ \pi \circ g = \alpha \circ g = h$.
\[
\begin{tikzcd}[column sep = 5pc]
G(X) \arrow[r, "\pi"] \arrow[dr, "\alpha"] & F(X,w) \arrow[d, "\overline{\alpha}"] \\
X \arrow[u, "g"] \arrow[r, "h"'] & A
\end{tikzcd}
\]

It is left to show uniqueness of $\overline{\alpha}$. Let $\gamma : F(X,w) \to A$ be a $\bal$-morphism satisfying $\gamma\circ f = h$. If $\alpha' = \gamma \circ \pi$, then $\alpha' : G(X) \to A$ is an $\lalg$-morphism and $\alpha' \circ g = \gamma \circ \pi \circ g = \gamma \circ f = h$. Since $G(X)$ is a free object in $\lalg$ and $\alpha' \circ g = h = \alpha \circ g$, uniqueness implies that $\alpha' = \alpha$. From this we get $\gamma \circ \pi = \alpha = \overline{\alpha} \circ \pi$. Because $\pi$ is onto, we conclude that $\gamma = \overline{\alpha}$.
\end{proof}

\begin{remark} \label{rem: norm of f(x)}
If $(X, w) \in \WSet$, then $\|f(x) \| = w(x)$. To see this, since $w : (X, w) \to (\mathbb{R}, \left| \cdot \right|)$ is a $\WSet$-morphism, by Theorem~\ref{thm: UMP holds}, there is a $\bal$-morphism $\alpha : F(X,w) \to \mathbb{R}$ with $\alpha \circ f = w$. Because $f$ is a weighted set morphism, by Lemma~\ref{lem: norm preserving} we have $w(x) = \left\|\alpha(f(x))\right\| \le \|f(x)\| \le w(x)$. Thus, $\|f(x)\| = w(x)$.
\end{remark}

We next show that the Yosida space $Y_{F(X,w)}$ of $F(X,w)$ is homeomorphic to a power of $[0,1]$, and that $F(X,w)$ embeds into the $\ell$-algebra of piecewise polynomial functions on $Y_{F(X,w)}$. For a set $Z$ we let $PP([0,1]^Z)$ be the $\ell$-algebra of piecewise polynomial functions on $[0,1]^Z$. If $Z$ is finite, then the definition of $PP([0,1]^Z)$ is standard (see, e.g., \cite[p.~651]{Del89}). If $Z$ is infinite, we define $PP([0,1]^Z)$ as the direct limit of $\{ PP([0,1]^Y) \mid Y \textrm{ a finite subset of }Z \}$. It is straightforward to see that $PP([0,1]^Z) \in \bal$.

For each $A \in \bal$ and $M \in Y_A$ it is well known that $A/M \cong \mathbb{R}$ (see, e.g., \cite[Cor.~2.7]{HJ61}). This allows us to identify the Yosida space $Y_A$ with the space $\hom_{\bal}(A,\mathbb{R})$ of $\bal$-morphisms from $A$ to $\mathbb{R}$, by sending $\alpha : A \to \mathbb{R}$ to $\ker(\alpha)$ and $M \in Y_A$ to the natural homomorphism $A \to \mathbb{R}$. The topology on $\hom_{\bal}(A,\mathbb{R})$ is the subspace topology of the product topology on $\mathbb{R}^A$.

\begin{theorem} \label{thm: Yosida}
Let $(X,w)\in\WSet$ and let $X' = \{ x \in X \mid w(x) > 0\}$.
\begin{enumerate}[$(1)$]
\item The Yosida space of $F(X,w)$ is homeomorphic to $[0,1]^{X'}$.
\item $F(X,w)$ embeds into $PP([0,1]^{X'})$.
\end{enumerate}
\end{theorem}

\begin{proof}
(1). We identify $Y_{F(X,w)}$ with $\hom_{\bal}(F(X,w), \mathbb{R})$ as in the paragraph before the theorem. From the universal mapping property, we see that there is a homeomorphism between $\hom_{\bal}(F(X,w), \mathbb{R})$ and $\hom_{\WSet}((X, w), (\mathbb{R}, \left|\cdot\right|))$. 
If $g : X \to \mathbb{R}$ is a $\WSet$-morphism, then $|g(x)| \le w(x)$, so $-w(x) \le g(x) \le w(x)$. Therefore, $\hom_{\WSet}((X, w), (\mathbb{R}, \left|\cdot\right|)) = \Pi_{x \in X} [-w(x), w(x)]$. If $x \in X'$, then $[-w(x), w(x)]$ is homeomorphic to $[0,1]$, and if $x \notin X'$, then $[-w(x), w(x)] = \{0\}$. Thus, $\Pi_{x \in X} [-w(x), w(x)]$ is homeomorphic to $[0,1]^{X'}$, and hence $Y_{F(X,w)}$ is homeomorphic to $[0,1]^{X'}$.

(2). Let  $\varphi : Y_{F(X,w)} \to \Pi_{x \in X'} [-w(x), w(x)]$ be the homeomorphism from the proof of (1) and let $\tau_x : [0,1] \to [-w(x), w(x)]$ be the homeomorphism given by $\tau_x(a) = 2w(x) a - w(x)$. If $\tau$ is the product of the $\tau_x$, then $\tau : [0,1]^{X'} \to \Pi_{x \in X'} [-w(x), w(x)]$ is a homeomorphism, and so $\rho := \tau^{-1}\circ \varphi$ is a homeomorphism from $Y_{F(X,w)}$ to $[0,1]^{X'}$. Therefore, $C(\rho) : C(Y_{F(X,w)}) \to C([0,1]^{X'})$ is a $\bal$-isomorphism. Since $F(X,w)$ is generated by $f[X]$, it is sufficient to show that $C(\rho)(f(x)) \in PP([0,1]^{X'})$. Let $x \in X$. If $w(x) = 0$, then since $\|f(x)\| = w(x)$ (see Remark~\ref{rem: norm of f(x)}), $f(x) = 0$, so $C(\rho)(f(x)) = 0 \in PP([0,1]^{X'})$. Suppose that $w(x) > 0$. Then $C(\rho)(f(x)) = 2w(x)p_x - w(x) \in PP([0,1]^{X'})$, completing the proof.
\end{proof}

\begin{remark}
We compare our results with those in the vector lattice literature. Recall (see, e.g., \cite[p.~48]{LZ71}) that the definition of a vector lattice, or Riesz space, is the same as that of an $\ell$-algebra except that multiplication is not present in the signature, and so in vector lattices there is no analogue of the multiplicative identity.
\begin{enumerate}
\item Let $\VL$ be the category of vector lattices and vector lattice homomorphisms. Then $\VL$ is a variety, so free vector lattices exist by Birkhoff's theorem (see, e.g., \cite[Thm.~10.12]{BS81}). Therefore, the forgetful functor $U : \VL \to \Set$ has a left adjoint.
\item Let a \emph{pointed vector lattice} be a vector lattice with a prescribed element, and a pointed vector lattice homomorphism a vector lattice homomorphism preserving the prescribed element. The associated category $\pVL$ is a variety, so the forgetful functor $U : \pVL \to \Set$ has a left adjoint.
\item If we consider the full subcategory $\uVL$ of $\pVL$ consisting of pointed vector lattices whose prescribed element is a strong order-unit, then Birkhoff's theorem does not apply since $\uVL$ is not a variety. In fact, an argument similar to the proof of Theorem~\ref{lem: no frees} shows that the forgetful functor $U : \uVL \to \Set$ does not have a left adjoint. However, a small modification of the proof of Theorem~\ref{thm: UMP holds} yields that the forgetful functor $U : \uVL \to \WSet$ does have a left adjoint.
\item Baker \cite[Thm.~2.4]{Bak68} showed that the free vector lattice $F(X)$ on a set $X$ embeds in the vector lattice $PL(\mathbb{R}^X)$ of piecewise linear functions on $\mathbb{R}^X$. In fact, Baker shows that $F(X)$ is isomorphic to the vector sublattice of $PL(\mathbb{R}^X)$ generated by the projection functions. Theorem~\ref{thm: Yosida}(2) is an analogue of Baker's result since the proof shows that $F(X,w)$ is isomorphic to the subalgebra of $PP([0,1]^{X'})$ generated by the projection functions. 
Beynon \cite[Thm.~1]{Bey74} showed that if $X$ is finite, then $F(X) = PL(\mathbb{R}^X)$. The analogue of Beynon's result for $\ell$-algebras is related to the famous Pierce-Birkhoff conjecture \cite[p.~68]{BK56} (see also \cite{Mah84, Mad89}).
\end{enumerate}
\end{remark}

\section{Some Consequences}

The proof of Theorem~\ref{lem: no frees} also yields that the forgetful functor $\ubal\to\Set$ does not have a left adjoint. On the other hand, since the forgetful functor $\bal\to\WSet$
has a left adjoint, if $\mathcal{C}$ is a reflective subcategory of $\bal$, then the forgetful functor $\mathcal{C} \to \WSet$ also has a left adjoint because the composition of adjoints is an adjoint \cite[Prop.~18.5]{AHS06}. Consequently, since $\ubal$ is a reflective subcategory of $\bal$, we obtain:

\begin{proposition} \label{prop: frees in ubal}
The forgetful functor $U : \ubal \to \WSet$ has a left adjoint.
\end{proposition}

Since taking uniform completion is the reflector $\bal\to\ubal$, the left adjoint of Proposition~\ref{prop: frees in ubal}
is obtained as the uniform completion of $F(X,w)$ for each $(X,w)\in\WSet$.

We next turn to describing a left adjoint to the forgetful functor $\balg \to \WSet$. We recall that an $\ell$-algebra $A$ is \emph{Dedekind complete} if each subset of $A$ that is bounded above has a least upper bound (and hence each subset bounded below has a greatest lower bound) in $A$. We also recall that if $A$ is a commutative ring with 1, then the set $\Id(A)$ of idempotents of $A$ is a boolean algebra under the operations
\[
e \vee f = e + f - ef, \quad e \wedge f = ef, \quad \lnot e = 1-e.
\]

\begin{definition} \cite[Def.~3.6]{BMO20b}
We call $A \in \bal$ a \emph{basic algebra} if $A$ is Dedekind complete and the boolean algebra $\Id(A)$ is atomic.
\end{definition}

Let $A,B$ be basic algebras. Following \cite[Def.~18.12]{LZ71}, we call a $\bal$-morphism $\alpha:A\to B$ a \emph{normal homomorphism} if it preserves all existing joins and meets. Let $\balg$ be the category of basic algebras and normal homomorphisms. Then $\balg$ is a non-full subcategory of $\bal$. The category $\balg$ was introduced in \cite{BMO20b} where it was shown that $\balg$ is dually equivalent to $\Set$, hence providing a ring-theoretic version of Tarski duality. Thus, $\balg$ plays a similar role in $\bal$ to that of $\caba$ in $\sf BA$.

The functors $B:\Set\to\balg$ and $X:\balg\to\Set$ establishing the dual equivalence between $\Set$ and $\balg$ are defined as follows.
For a set $X$ let $B(X)$ be the $\ell$-algebra of all bounded real-valued functions, and for a map $\varphi : X\to Y$ let $B(\varphi) : B(Y)\to B(X)$ be given by $B(\varphi)(f)=f\circ\varphi$ for $f\in B(Y)$. Then $B:\Set\to\balg$ is a well-defined contravariant functor.

For $A\in\balg$ let $X_A$ be the set of atoms of $\Id(A)$. We then set $X(A)=X_A$, and for a $\balg$-morphism $\alpha:A \to B$ we let $X(\alpha):X_B \to X_A$ be given by
\[
X(\alpha)(x) = \bigwedge \{ a\in \Id(A) \mid x \le \alpha(a) \}
\]
for $x\in X_A$.
Then $X : \balg\to\Set$ is a well-defined contravariant functor, and the functors $B$ and $X$ yield a dual equivalence of $\balg$ and $\Set$. The natural isomorphisms $\eta : 1_{\Set} \to X \circ B$ and $\vartheta : 1_{\balg} \to B \circ X$ are defined by letting $\eta_X(x)$ be the characteristic function of $\{x\}$ for each $x\in X$, and
\[
\vartheta_A(a)(x)=\zeta_A(a)((1-x)A) \mbox{ for each } a\in A \mbox{ and } x\in X_A,
\]
where $(1-x)A$ is the $\ell$-ideal of $A$ generated by $1-x$ (it is maximal since $x$ is an atom of $\Id(A)$).

As was shown in \cite{BMO18c}, for $A\in\bal$, the $\ell$-algebra $B(Y_A)$ together with $\zeta_A:A\to B(Y_A)$ is the (unique up to isomorphism) canonical extension of $A$, where we recall that a {\em canonical extension} of $A$ is $A^\sigma\in\balg$ together with a $\bal$-monomorphism $e:A\to A^\sigma$ satisfying:
\begin{enumerate}
\item (Density) Each $x\in A^\sigma$ is a join of meets of elements of $e[A]$.
\item (Compactness) For $S,T\subseteq A$ and $0<\varepsilon\in\mathbb R$, from $\bigwedge e[S] + \varepsilon \le \bigvee e[T]$ it follows that $\bigwedge e[S'] \le \bigvee e[T']$ for some finite $S'\subseteq S$ and $T'\subseteq T$.
\end{enumerate}

\begin{theorem} \label{prop: sigma is a reflector}
$(\cdot)^\sigma:\bal\to\balg$ is a reflector, so $\balg$ is a (non-full) reflective subcategory of $\bal$.
\end{theorem}

\begin{proof}
Let $A \in \bal$, $C \in \balg$ and $\alpha : A\to C$ be a $\bal$-morphism. By \cite[p.~89]{Mac71}, it suffices to show that there is a unique $\balg$-morphism $\gamma : A^\sigma \to C$ with $\gamma \circ e = \alpha$. Since $\alpha$ is a $\bal$-morphism, $Y(\alpha) : Y_C \to Y_A$ is a continuous map. Let $f : X_C \to Y_A$ be given by $f(x)=Y(\alpha)((1-x)C)$ for each $x\in X_C$. In other words, if we identify $X_C$ with a subset of $Y_C$ (by sending $x$ to $(1-x)C$), then $f$ is the restriction of $Y(\alpha)$ to $X_C$.
This induces a $\bal$-morphism $B(f)$ from $A^\sigma = B(Y_A)$ to $B(X_C)$. 
Since $\vartheta_C : C \to B(X_C)$ is an isomorphism, 
we have a $\balg$-morphism $\gamma := \vartheta_C^{-1} \circ B(f) : B(Y_A) \to C$.
\[
\begin{tikzcd}[column sep = 5pc]
A \arrow[r, "e"] \arrow[d, "\alpha"'] & B(Y_A) \arrow[d, "B(f)"] \arrow[dl,  "\gamma"] \\
C \arrow[r, "\vartheta_C"'] & B(X_C)
\end{tikzcd}
\]
We show that $\gamma \circ e = \alpha$. For this it suffices to show that $B(f) \circ e = \vartheta_C \circ \alpha$. Let $x \in X_C$ and $a \in C$. Then $B(f)(e(a)) = e(a) \circ f$ sends $x$ to 
$\zeta_A(a)(\alpha^{-1}((1-x)C))$, which is equal to the unique $r \in \mathbb{R}$ satisfying $a + \alpha^{-1}((1-x)C) = r + \alpha^{-1}((1-x)C)$.
On the other hand, 
\[
(\vartheta_C \circ \alpha)(a)(x) = \vartheta_C(\alpha(a))(x) = \zeta_C(\alpha(a))((1-x)C),
\]
which is the unique $s \in \mathbb{R}$ satisfying $\alpha(a) + (1-x)C = s + (1-x)C$. Since $a - r \in \alpha^{-1}((1-x)C)$, we have $\alpha(a - r) \in (1-x)C$. Therefore, $\alpha(a) - r \in (1-x)C$, so $\alpha(a) + (1-x)C = r + (1-x)C$. Thus, $r = s$, and hence $B(f) \circ e(a)$ and $(\vartheta_C \circ \alpha)(a)$ agree for each $x \in X_C$. Since $a \in C$ was arbitrary, we conclude that $B(f) \circ e = \vartheta_C \circ \alpha$.

For uniqueness, suppose that $\gamma' : A^\sigma \to C$ satisfies $\gamma' \circ e = \alpha$. Then $\gamma'|_{e[A]} = \gamma|_{e[A]}$. Since $\gamma$ and $\gamma'$ are $\balg$-morphisms and $e[A]$ is dense in $A^\sigma$, we conclude that $\gamma' = \gamma$.
\end{proof}

The following is now an immediate consequence of Theorems~\ref{thm: UMP holds} and~\ref{prop: sigma is a reflector}.

\begin{proposition} \label{prop: frees in balg}
The forgetful functor $U : \balg \to \WSet$ has a left adjoint.
\end{proposition}

This left adjoint is obtained as the canonical extension of $F(X,w)$ for each $(X,w)\in\WSet$. On the other hand, the proof of Theorem~\ref{lem: no frees} shows that the forgetful functor $\balg\to\Set$ does not have a left adjoint.

\def\cprime{$'$}
\providecommand{\bysame}{\leavevmode\hbox to3em{\hrulefill}\thinspace}
\providecommand{\MR}{\relax\ifhmode\unskip\space\fi MR }
\providecommand{\MRhref}[2]{%
  \href{http://www.ams.org/mathscinet-getitem?mr=#1}{#2}
}
\providecommand{\href}[2]{#2}

\end{document}